\documentclass[11pt
]{amsart}

\usepackage{amssymb}

\usepackage{hyperref}
\hypersetup{bookmarksdepth=3}

\usepackage[hide]{ed}

    \normalbaselines                          %

\renewcommand{\labelenumi}{(\roman{enumi})}

\newcounter{vremennyj}

\newcommand\cond[1]{\setcounter{vremennyj}{\theenumi}\setcounter{enumi}{#1}\labelenumi\setcounter{enumi}{\thevremennyj}}

\newcommand{\R}{{\mathbb  R}}
\newcommand{\s}{{\mathbb  S}}

\newcommand{\Z}{{\mathbb  Z}}
\newcommand{\N}{{\mathbb  N}}
\newcommand{\C}{{\mathbb  C}}

\newcommand{\cS}{\mathcal{S}}

\newcommand{\fdot}{\,\cdot\,}

\newcommand{\wt}{\widetilde}

\newcommand{\1}{\mathbf{1}}
\newcommand{\cM}{\mathcal{M}}
\newcommand{\cB}{\mathcal{B}}

\newcommand{\f}{\varphi}

\newcommand{\e}{\varepsilon}

\newcommand{\cz}{Cald\-er\-\'{o}n--Zygmund\ }

\DeclareMathOperator{\supp}{supp}
\DeclareMathOperator{\dist}{dist}
\DeclareMathOperator{\diam}{diam}

\newcommand{\bj}{\mathbf{j}}

\newcommand{\La}{\langle}
\newcommand{\Ra}{\rangle}

\newcommand{\ci}[1]{_{ {}_{\scriptstyle #1}}}
\newcommand{\ti}[1]{_{\scriptstyle \text{\rm #1}}}

\catcode`@=11
\chardef\mathlig@atcode\count255

\def\actively#1#2{\begingroup\uccode`\~=`#2\relax\uppercase{\endgroup#1~}}
\def\mathlig@gobble{\afterassignment\mathlig@next@cmd\let\mathlig@next= }
\def\mathlig@delim{\mathlig@delim}
\def\mathlig@defcs#1{\expandafter\def\csname#1\endcsname}
\def\mathlig@let@cs#1#2{\expandafter\let\expandafter#1\csname#2\endcsname}
\def\mathlig@appendcs#1#2{\expandafter\edef\csname#1\endcsname{\csname#1\endcsname#2}}

\def\mathlig#1#2{\mathlig@checklig#1\mathlig@end\mathlig@defcs{mathlig@back@#1}{#2}\ignorespaces}

\def\mathlig@checklig#1#2\mathlig@end{%
\expandafter\ifx\csname mathlig@forw@#1\endcsname\relax
\expandafter\mathchardef\csname mathlig@back@#1\endcsname=\mathcode`#1%
\mathcode`#1"8000\actively\def#1{\csname mathlig@look@#1\endcsname}%
\mathlig@dolig#1\mathlig@delim
\fi
\mathlig@checksuffix#1#2\mathlig@end
}

\def\mathlig@checksuffix#1#2\mathlig@end{%
\ifx\mathlig@delim#2\mathlig@delim\relax\else\mathlig@checksuffix@{#1}#2\mathlig@end\fi
}
\def\mathlig@checksuffix@#1#2#3\mathlig@end{%
\expandafter\ifx\csname mathlig@forw@#1#2\endcsname\relax\mathlig@dosuffix{#1}{#2}\fi
\mathlig@checksuffix{#1#2}#3\mathlig@end
}

\def\mathlig@dosuffix#1#2{%
\mathlig@appendcs{mathlig@toks@#1}{#2}%
\mathlig@dolig{#1}{#2}\mathlig@delim
}

\def\mathlig@dolig#1#2\mathlig@delim{%
\mathlig@defcs{mathlig@look@#1#2}{%
\mathlig@let@cs\mathlig@next{mathlig@forw@#1#2}\futurelet\mathlig@next@tok\mathlig@next}%
\mathlig@defcs{mathlig@forw@#1#2}{%
  \mathlig@let@cs\mathlig@next{mathlig@back@#1#2}%
  \mathlig@let@cs\checker{mathlig@chck@#1#2}%
  \mathlig@let@cs\mathligtoks{mathlig@toks@#1#2}%
  \expandafter\ifx\expandafter\mathlig@delim\mathligtoks\mathlig@delim\relax\else
  \expandafter\checker\mathligtoks\mathlig@delim\fi
  \mathlig@next
}%
\mathlig@defcs{mathlig@toks@#1#2}{}%
\mathlig@defcs{mathlig@chck@#1#2}##1##2\mathlig@delim{%
  \ifx\mathlig@next@tok##1%
   \mathlig@let@cs\mathlig@next@cmd{mathlig@look@#1#2##1}\let\mathlig@next\mathlig@gobble
  \fi
  \ifx\mathlig@delim##2\mathlig@delim\relax\else
   \csname mathlig@chck@#1#2\endcsname##2\mathlig@delim
  \fi
}%
\ifx\mathlig@delim#2\mathlig@delim\else
  \mathlig@defcs{mathlig@back@#1#2}{\csname mathlig@back@#1\endcsname #2}%
\fi
}%

\catcode`@\mathlig@atcode

\mathchardef\ordinarycolon\mathcode`\:
\def\vcentcolon{\mathrel{\mathop\ordinarycolon}}
\mathlig{:=}{\vcentcolon=}
\mathlig{::=}{\vcentcolon\vcentcolon=}

\numberwithin{equation}{section}

\theoremstyle{plain}
\newtheorem{theo}{Theorem}[section]
\newtheorem{cor}[theo]{Corollary}
\newtheorem{lem}[theo]{Lemma}
\newtheorem{prop}[theo]{Proposition}

\theoremstyle{definition}
\newtheorem{defn}[theo]{Definition}
\newtheorem*{defn*}{Definition}

\theoremstyle{remark}
\newtheorem{ex}[theo]{Example}
\newtheorem*{ex*}{Example}
\newtheorem{rem}[theo]{Remark}
\newtheorem*{exs*}{Examples}
\newtheorem*{rems*}{Remarks}
\theoremstyle{remark}
\newtheorem*{rem*}{Remark}

\title[Regularizations of singular integral operators]{Regularizations of  general singular integral operators}
\author{Constanze~Liaw}
\address{Department of Mathematics, Texas A\&M University, Mailstop 3368,      
College Station, TX  77843, USA }
\email{conni@math.tamu.edu}
\urladdr{http://www.math.tamu.edu/\~{}conni}
\author{Sergei~Treil}
\address{Department of Mathematics, Brown University, 151 Thayer
Str./Box 1917,      
Providence, RI  02912, USA }
\email{treil@math.brown.edu}
\urladdr{http://www.math.brown.edu/\~{}treil}

\begin{document}

\begin{abstract}
In the theory of singular integral operators significant effort is often required to rigorously define  such an operator. This is due to the fact that the kernels of such operators are not locally integrable on the diagonal $s=t$, so the integral formally defining the operator $T$ or its bilinear form $\La Tf, g \Ra$ is not well defined (the integrand in not in $L^1$) even for ``nice'' $f$ and $g$.  

However, since the kernel only has singularities on the ``diagonal'' $s=t$, the bilinear form $\La Tf, g \Ra$ is well defined say for bounded compactly supported functions with separated supports. 

One of the standard ways to interpret the boundedness of a singular integral operators is to consider the \emph{regularized kernel} $K_\e(s, t) = K(s, t) m((s-t)/\e) $, where the cut-off function $m$ is $0$ in a neighborhood of the origin, so the integral operators $T_\e$ with kernel $K_\e$ are well defined (at least on a dense set). Then instead of asking about the boundedness of the operator $T$, which is not well defined, one can ask about uniform boundedness (in $\e$) of the regularized operators $T_\e$. 

For the standard regularizations one usually considers \emph{truncated} operators $T_\e$ with $m(s) =\1\ci{[1, \infty)} (|s|)$, although smooth cut-off functions were also considered in the literature. 

The main result of the paper is that for a wide class of singular integral operators (including the classical \cz operators in non-homogeneous two weight settings), the so-called restricted $L^p$ boundedness, i.e.~the uniform estimate 
\[
|\La Tf, g \Ra| \le C \|f\|_p \|g\|_{p'}
\]
for bounded compactly supported $f$ ans $g$ with separated supports implies the uniform $L^p$-boundedness of regularized operators $T_\e$ for any reasonable choice of a smooth cut-off function $m$.   For example, any $m\in C^\infty(\R^N)$, $m\equiv 0$ in a neighborhood of $0$, and such that $1-m$ is compactly supported would work. 

If the kernel $K$ satisfies some additional assumptions (which are satisfied for classical singular integral operators like Hilbert Transform, Cauchy Transform, Ahlfors--Beurling Transform, Generalized Riesz Transforms), then the restricted $L^p$ boundedness also implies the uniform $L^p$ boundedness of the classical truncated operators $T_\e$ ($m(s) =\1\ci{[1, \infty)} (|s|)$). 
\end{abstract}

\maketitle

\section{Introduction}\label{CR}


\subsection{Preliminaries}
\label{s.prelim}
Generally,  a singular integral operator is understood as an operator $T$ on $L^2(\mu)$ (or on $L^p(\mu)$) that is \emph{formally} given by
\begin{equation}
\label{sio-1}
Tf(s)=\int K(s,t) f(t)\, d\mu(t),
\end{equation}
where the kernel $K(s,t)$ is singular near $s=t$, i.e.~$K(s,\fdot)$ and $K(\fdot,t)$ are not in $L^1\ti{loc}$ near the singularity. This means that the above integral is not defined even for the simplest functions $f$ (which explains the word \emph{formally} above), and the question of how to interpret this expression immediately arises.

In  simple cases the interpretation is quite easy. For example, if $T$ is the classical Hilbert Transform on the real line ($K(s, t) = \pi^{-1}(s-t)^{-1}$, $\mu$ is the Lebesgue measure on $\R$), it is an easy exercise to show that for a compactly supported smooth function  $f$ the integral \eqref{sio-1} exists in the sense of principal value, i.e.~that
\begin{equation}
\label{pv-hilb}
\lim_{\alpha\to 0+} \int_{|s-t|>\alpha} \frac{f(t)}{s-t} dt
\end{equation}
exists (for all $s$, provided that $f$ is $C^1$ and compactly supported). Thus, the operator $T$ is well defined on a dense set, and if one proves $L^2$ (or $L^p$) bounds on this set, the operator extends by continuity to all $L^2$ (resp.~$L^p$). For the Hilbert transform the $L^p$ estimates are a classical and well known result, so the Hilbert transform is a well defined
bounded operator in $L^p$, $p\in (1, \infty)$.
\footnote{It is also known, although significantly harder to prove,  that the principal value \eqref{pv-hilb} exists a.e.~for all compactly supported $L^1$ functions $f$, which immediately implies the existence of the principal value for all $f\in L^p$, $p\in(1, \infty)$.}

Such a na\"\i ve approach also works for other ``nice'' classical singular integral operators, like Riesz Transforms in $\R^n$. However, the situation becomes more complicated, if one considers more general measures and/or kernels: the existence of principal values in such situations is far from trivial, and usually requires a lot of effort.

So, given a (formal)  singular integral operator, how can one define it and investigate whether it is bounded in $L^p$?  One of the standard approaches in the general situation is to consider truncated operators
$T_\e$,
\begin{equation}
\label{trunc-op}
T_\e f(s) = \int_{|s-t|>\e} K(s,t) f(t) d\mu(t)
\end{equation}
which are well defined for compactly supported $f$ if one assumes, for example, that $K$ is locally bounded off the ``diagonal'' $s=t$. In this case, the boundedness in $L^p$ is defined as the uniform boundedness of the truncated operators $T_\e$; if the operators $T_\e$ are uniformly bounded, one can then take a  limit point (in the weak operator topology) of $T_\e$, $\e\to 0$ and the corresponding singular integral operator $T$. Note that the weak limit point does not have to be unique.

Moreover, if the operators $T_\e$ are uniformly bounded, then it is often possible to prove the existence of principal values (at least for $f$ in a dense set), so one can define the singular integral operator in a natural fashion.

Instead of the truncations, one can also consider smooth regularizations of the kernel $K$. For example, for the Hilbert transform, it is very natural to move to the complex plane and consider operators $T_\e$,
\[
T_\e f(s) = \frac1\pi \int_{\R} \frac{f(t)}{s-t+i\e} dt .   
\]

Further, there is an alternative, ``axiomatic'' way  of defining a singular integral operator with kernel $K$, see for example \cite{christlectures}. Namely, we assume that we are given an operator, or, more precisely its bilinear form $\langle T f, g\rangle$, well defined on some smaller set (the Schwartz class or the class of $C^\infty$ functions with compact support are often used). The statement that $T$ is an integral operator with kernel $K$ means simply that
\[
\langle T f, g\rangle = \int K(s,t) f(t) g(s) d\mu(t) d\mu(s)
\]
for $f$ and $g$ with separated compact supports (so the above integral is well defined). In many cases it was shown that if the operator $T$ is bounded in $L^p$, then the truncated operators $T_\e$ are uniformly  bounded as well. Note, that in the above abstract approach we require some kind of \emph{apriori bounds} on the operator, because we assume that its bilinear form is well defined on some smaller space.

Let us also mention, that in the theory of \cz operators (see the definition of a \cz operator below), if the kernel $K$ is antisymmetric, $K(s, t) = -K(t, s)$, there is a canonical way to interpret the operator without any apriori boundedness assumptions, see for example \cite{christlectures} for the homogeneous case and  \cite{NTV} for the non-homogeneous case.

Namely, antisymmetry means that formally $T^*=-T$, so $\langle Tf , g\rangle = - \langle f , Tg\rangle$, so (again formally)
\begin{align}
\notag
\langle Tf , g\rangle & = \int K(s, t) f(t) g(s) d\mu(t) d\mu(s) = - \int K(s, t) f(s) g(t) d\mu(t) d\mu(s)
\\
\label{sym-czo}
& = \frac12  \int K(s, t) \left[ f(t) g(s) - f(s) g(t) \right] d\mu(t) d\mu(s).
\end{align}
If  $K$ is a \cz kernel of dimension $d$ in $\R^N$,  then by the definition $|K(s, t)|\le C |s-t|^{-d}$.  
Therefore if $f$ and $g$ are compactly supported $C^1$ functions, then the integrand in \eqref{sym-czo} can be estimated by $C|s-t|^{-d+1}$.

For such kernels it is usually assumed that the measure satisfies the condition $\mu(\{x:|x-x_0|<r\}) \le C r^d$ for all $x_0$ and $r$ (this condition is usually necessary for the $L^2$ boundedness of classical \cz operators of dimension $d$, like Cauchy transform in $\C$).  But for such measures $\int_Q |s-t|^{-d+1} d\mu(s)d\mu(t)<\infty$ for any compact $Q$, so $\langle Tf , g\rangle$ is well defined for  $f, g \in C^1_0$.  

\subsection{Description of the main results}
\label{ss:main-results} 
The main result of the paper can be stated in one sentence as ``the situation with interpretation of singular integral operators is much simpler, than it seems; to investigate the boundedness one only needs to study an elementary and well defined \emph{restricted} bilinear form''.

The main idea is embarrassingly simple, and we should be ashamed that we did not arrive to it much earlier, although some preliminary results in this direction were obtained by us in \cite{mypaper}, and   in thesis of the first author.  In our defense we can only say that this idea was overlooked by generations of harmonic analysts before us. 

Let us describe the main results in more detail. We will need some definitions.

\begin{defn}
Let Radon measures $\mu$ and $\nu$ in $\R^N$ be fixed. A singular kernel in $\R^N$ (with respect to the measures $\mu$ and $\nu$) is a $\mu\times\nu$-measurable function $K$ on $\R^N\times \R^N$, which is locally $L^2(\mu\times \nu)$ off the diagonal $\{(s, t)\in \R^N\times \R^N : s=t\}$. 

We say that the singular kernel $K$ is a singular kernel of order $d$ if the kernel $\wt K$
\[
\wt K(s, t) = \left\{
\begin{array}{ll}
K(s, t) |s-t|^d, \qquad & s\ne t \\
0   & s=t
\end{array} \right.
\]
is locally $L^2(\mu\times \nu)$.  Note, that for any singular kernel on $\R^N$ the kernel $\wt K$ is locally $L^2(\mu\times\nu)$ off the diagonal, so one only needs to check this condition on the diagonal. 
\end{defn}

For a singular kernel  $K$  in $\R^N$ (with respect to Radon measures $\mu$ and $\nu$) the expression 
\[
\langle Tf, g\rangle = \int K(s, t) f(t) g(s) d\mu(t) d\nu(s)
\]
is well defined for all Borel measurable bounded functions $f$, $g$ with separated compact supports ($\dist(\supp f, \supp g)>0$).

\begin{defn}
\label{d:r-bound}
Let $p\in (1, \infty)$ and Radon measures $\mu$ and $\nu$ on $\R^N$ be fixed. Let $K$ be a singular kernel in $\R^N$. 
We say that the formal  singular integral operator  with the kernel $K$ is restrictedly bounded in $L^p$ (i.e.~as an operator $T:L^p (\mu) \to L^p(\nu)$) if 
\begin{equation}
\label{restr-bound}
\left| \int K(s, t) f(t) g(s) d\mu(t) d\nu(s)  \right| \le C \|f\|_{L^p(\mu)} \|g\|_{L^{p'}(\nu)}, \qquad 1/p + 1/p' =1
\end{equation}
for all bounded $f$, $g$ with separated compact supports.

Sometimes,  abusing the language, we will just  say that the kernel $K$ is restrictedly $L^p$ bounded with bound $C$. 

The least constant $C$ in \eqref{restr-bound} (with $p$, $\mu$ and $\nu$ fixed) is called the restricted norm of $K$.
\end{defn}

It is easy to check that for any fixed $p$, $\mu$, $\nu$ the restricted norm is a (semi)norm on the set of singular kernels. 

\begin{rem*}
For fixed measures $\mu$ and $\nu$ one can always assume that the kernel $K$ is defined only a.e.~with respect to the measure $\mu\times \nu$. It is not hard to  show that if the measures $\mu$ and $\nu$ do not have  common atoms (which we usually assume), then the diagonal $\{(s, s):s\in\R^N\}$ of  $\R^N\times \R^N$ has $\mu\times\nu$ measure $0$, so the values of $K$ on the diagonal do not have to be specified. 
\end{rem*}

Our first main result is that 
there are many families of    smooth mollifying multipliers 
%
%
$M_\e(s, t)$,  such that
\begin{enumerate}
\item $M_\e (s, t)\to 1$ as $\e\to 0$ uniformly on all sets $\{s, t \in\R^N: |s-t| >a\}$, $a>0$; 
  \item For any singular kernel $K$ (with respect to Radon measures $\mu$ and $\nu$) the regularized kernels $K_\e = K M_\e$ are locally in $L^2(\mu\times\nu)$, so the corresponding operators are well defined for bounded compactly supported functions;
    \item If the kernel is restrictedly bounded in $L^p$ (i.e.~if the estimate \eqref{restr-bound} holds for all bounded $f$, $g$ with separated compact supports), and the measures $\mu$ and $\nu$ do not have common atoms, then the regularized integral operators $T_\e$ with kernels $K_\e$ are uniformly (in $\e$) bounded as operators $L^p(\mu) \to L^p(\nu)$. 
\end{enumerate}

In this case, one can take the limit point $T$ (in the weak operator topology) of $T_\e$, $\e\to 0+$ as a singular integral operator with the kernel $K$. It is easy to see that
\begin{equation}
\label{abstr-sio}
\langle Tf,g \rangle = \int K(s, t) f(t) g(s) d\mu(t) d\nu(s)
\end{equation}
for all bounded $f$ and $g$ with separated compact supports, so $T$ is indeed a singular integral operator with kernel $K$ in the sense of the abstract ``axiomatic'' approach, described above in Section \ref{s.prelim}.

Note, that such limit point does not need to be unique, but it is easy to show that the difference between any two bounded  operators $L^p(\mu) \to L^p(\nu)$ satisfying \eqref{abstr-sio} (with the same kernel $K$) is always a multiplication operator.  

A simple way to construct a mollifying multiplier  is to take an arbitrary $C^\infty$ function $m$, $m\equiv 0$ in a neighborhood of the origin, and such that $1-m$ has a compact support,  and define $M_\e(s, t):= m((t-s)/\e)$. If at the origin one only requires that $|m(x)| \le C |x|^d$, then the function $M_\e(s, t):= m((t-s)/\e)$ will regularize the kernels of order up to $d$. 

Next, we will show that, under additional assumptions  on the kernel $K$, the restricted boundedness of $K$ implies the uniform boundedness of the truncated operators \eqref{trunc-op}.  

And finally, we will show that under some additional assumption, the restricted boundedness implies a two weight Muckenhoupt condition,  see Theorem \ref{t:Muckehnoupt} below for the exact statement.

\section{Smooth regularizations of singular kernels}\label{sec2}

\subsection{A trivial idea}
We start with a simple observation that if a kernel $K(s, t)$, $s, t\in \R^n$ is restrictedly bounded in $L^p$ with a bound $C$, then for any $a\in \R^N$ the kernel
\[
K(s, t) e^{-a\cdot t} e^{a\cdot s}
\]
is also restrictedly bounded with the same constant. This follows from the trivial fact that a multiplication by a unimodular function is always an invertible isometry in all $L^p(\mu)$, and that it does not change the support.

Averaging $K(s, t) e^{-i a\cdot t} e^{ ia\cdot s}$ over all $a\in\R^N$ with weight $\rho$, $\rho\ge 0$, $\int_{\R^N} \rho(a) da =1$ we get that the averaged  kernel
\[
\int_{\R^N} \rho (a) K(s, t) e^{-i a\cdot t} e^{i a\cdot s} da
\]
is also restrictedly bounded with the same constant $C$. Note, that we do not even have to assume that $\rho\ge0$. It is sufficient to assume that $\rho\in L^1(dx)$ ($L^1$ with respect to Lebesgue measure); in this case the averaged kernel will be bounded with constant $C \|\rho\|_1$, where $\|\fdot\|_1$ is  the $L^1$ norm with respect to the Lebesgue measure.

One can immediately see that
\[
\int_{\R^N} \rho (a) K(s, t) e^{-i a\cdot t} e^{i a\cdot s} da = \widehat \rho(t-s) K(s,t)
\]
where $\widehat \rho$ denotes the Fourier Transform, $\widehat \rho (s) = \int_{\R^N} \rho(x) e^{-i s\cdot x} dx$.

We can summarize the above reasoning in the following lemma.

\begin{lem}
\label{l-reg1}
Let $K$ be a restrictedly $L^p$ bounded kernel (i.e.~estimate \eqref{restr-bound} holds for all bounded compactly supported functions with separated supports) with a bound $C$. Assume that  $\rho\in L^1(dx)$ and let
\[
M=1- \widehat\rho, \qquad M_\e(x) := M(x/\e).
\]
Then the kernels $K_\e( s, t):= K(s, t) M_\e(t-s)$ are $L^p$ restrictedly bounded with constant $(1+ \|\rho\|_1) C$.
\end{lem}

\begin{proof}
The estimate for $\e=1$ was already explained above. To proof the estimate for general $\e$ one just can notice that $\widehat\rho(s/\e)$ is the Fourier transform of the function $\e^{N} \rho(\e x)$ and
\[
\int_{\R^N} | \e^{N} \rho(\e x) | dx = \int_{\R^N} |  \rho( x) | dx  = \|\rho\|_1.
\]
\end{proof}

\begin{lem}
For the function $M$ (and $M_\e$) defined in the previous lemma, the following holds:
\begin{enumerate}
    \item For any $a>0$, the function $M_\e(s)\to 1$ as $\e\to0+$ uniformly on the set $\{s\in\R^N, |s|>a\}$.
    
    \item If $\widehat\rho\in C^1$ (for example, if $\int_{\R^N} (1+|x|)|\rho(x)| dx<\infty$) and if $\int_{\R^N} \rho(x) dx =1$, then for any $\e>0$
    \[
    M_\e (s) = O(|s|)\qquad \text{as } s\to 0.
    \]
    
    \item If  $\int_{\R^N} (1+|x|^2)|\rho(x)| dx<\infty$, and $\int_{\R^N} \rho(x) dx =1$, and
    \[
    \int_{\R^N} x_k \rho(x) dx =0, \qquad \forall k=1, 2, \ldots, N,
    \]
    then
    \[
    M_\e (s) = O(|s|^2)\qquad \text{as } s\to 0.
    \]
    
    \item Moreover, if $\int_{\R^N} \rho(x) dx =1$ and if $\int_{\R^N} (1+|x|^k)|\rho(x)| dx<\infty$ for some $k\in \N$, and  
    \[
    \int_{\R^N} x^\alpha \rho(x) dx =0,  
    \]
    for all multi-indices $\alpha$, $|\alpha|< k$, then
    \[
    M_\e (s) = O(|s|^{k})\qquad \text{as } s\to 0.
    \]
\end{enumerate}
\end{lem}
\begin{proof}
The proof follows from the basic properties of the Fourier Transforms, we leave it as an exercise for a reader. Statement \cond1 follows, for example, from the Riemann--Lebesgue Lemma.
\end{proof}

\subsection{Some examples}

\begin{ex}
On the real line $\R$ consider the weight $\rho(x) = e^{-x} \1_{[0, \infty)}(x)$. For this weight $\widehat \rho(s) = (1+is)^{-1}$, so the mollifying factor $M(s)= 1- \widehat\rho(s)$ is given by
\[
M(s) =\frac{is}{1+is}= \frac{s}{s - i}, \qquad\text{so}\quad M_\e(s) := M(s/\e) = \frac{s}{s-i\e}
\]
For the Hilbert Transform kernel $K(s, t) = \pi^{-1} (s-t)^{-1}$ the regularization with this mollifying factor give as the kernel
\[
K_\e(s, t) = \frac1\pi\cdot \frac1{s-t} M_\e(t-s) = \frac1\pi \cdot \frac1{s-t} \cdot\frac{t-s}{t-s-i\e} = \frac1\pi\cdot \frac1{s-t +i\e}
\]  
which has the very natural complex analytic interpretation. That regularization is widely used in complex analysis, and our investigation of this regularization in \cite{mypaper} lead us to the main idea of this paper.
\end{ex}

\begin{ex}
Define the weight $\rho$ on $\R^N$ by $\rho(x) = (2\pi)^{-N/2} e^{-|x|^2/2}$. The Fourier Transform of $\rho$ is given by $\widehat\rho(s) = e^{-|s|^2/2}$, so the mollifying multiplier is $M(s) = 1- e^{-|s|^2/2}$. The regularized kernel $K_\e$ will be
\[
K_\e(s, t) := K(s, t) \left[1- e^{-|s-t|^2/2\e} \right] .
\]
Since $M$ has zero of order 2 at 0, this function will regularize singular kernels of order $d\le 2$ (i.e.~such that $|s-t|^d K(s, t)$ is locally bounded).  
\end{ex}

The mollifying multiplier $M$ regularizes only singular kernels of order $d\le 2$. To regularize kernels of higher order one can use, for example, its power $M_\e^k$, where the exponent $k$ is an integer, $k\ge N/2$. Applying Lemma \ref{l-reg1} $k$ times we get that given an $L^p$ restrictedly bounded kernel $K$ with a bound $C$, then $K(s,t) M_\e(t-s)^k$ is also restrictedly bounded with constant $2^k$.

However, we neither have to do that, nor do we need to be tied to a particular regularization. Using basic facts about Fourier Transform, we can construct mollifying multipliers without explicitly defining the weight and computing its Fourier Transform.

\subsection{Schur multipliers, Wiener algebra and Sobolev spaces}
Let us first introduce some definitions.  

\begin{defn}
Let $p\in (1, \infty)$ and Radon measures $\mu$ and $\nu$ in $\R^N$ be fixed. 

We say that a function $M$ on $\R^N \times \R^N$ is a \emph{Schur multiplier} on the set of  restrictedly $L^p$ bounded singular kernels, if
the map $K\mapsto KM$ is a bounded map on this set. In other words, $M$ is a Schur multiplier on this set if, there exists a constant $C_1$ such that  for any $L^p$ restrictedly bounded  kernel $K$ with a bound $C$, see Definition \ref{d:r-bound}, the kernel $MK$ is also restrictedly bounded  with constant $C_1 C$. 

The best possible constant $C_1$ (for a fixed $p$, $\mu$, and $\nu$) is called ($L^p$, $\mu$, $\nu$) restricted \emph{Schur norm} of $M$. 
\end{defn}

\begin{rem*}
Let us explain the terminology a little bit. The Schur product $A\circ B$ of two matrices is their entrywise  product, $(A\circ B)_{j,k} = a_{j,k}b_{j,k}$. 

Similarly, the Schur product of two kernels, $K(s, t)$ and $M(s, t)$ is their usual product (of two functions). The special term ``Schur product'' is sometimes used to distinguish it from the product (composition) of the corresponding integral operators. 

Let $X$ be a space of operators (like the space of bounded operators, or the Schatten--Von-Neumann class $\mathfrak S_p$). 
On the set of kernels one can introduce the norm inherited from the space $X$ of operators, the norm $\|K\|_X $ of a kernel is simply the norm of the corresponding integral operator in the space $X$. 

A function $M$ is called a Schur multiplier for the class $X$ if the map $K \mapsto MK$ is bounded with respect to the norm $\|K\|_X$.  

We use the same term ``Schur multiplier'', because our definition is very close in spirit to the classical one. 

We should also mention that while our definition definitely depends on $p$, $\mu$ and $\nu$, the Schur multipliers we construct below will be the \emph{universal} ones: they are Schur multipliers with uniform estimate on the Schur norm for all $p$, $\mu$, $\nu$, and also for all reasonable classes of operators like the bounded operators, $\mathfrak S_p$. 
\end{rem*}

Recall, that the Wiener algebra $W(\R^N)$ is the set of all $f =\widehat h$, $h\in L^1(\R^N)$, with the norm $\|f\|_W = \|h\|_1$. 

The reasoning in the beginning of this section can be summarized in the following lemma.

\begin{lem}
\label{l:wiener-schur-1}
Let $M\in W=W(\R^N)$. Then the function $\wt M$ on $\R^N\times \R^N$ defined by $\wt M(s, t) = M(t-s) $ is a Schur multiplier with Schur norm at most $\|M\|_W$. 
\end{lem}

The next trivial and well-known lemma gives a simple sufficient condition for $M\in W$. 
\begin{lem}
\label{l:sob-wiener}
Let $M$ belong to the Sobolev space $W^{k, 2}(\R^N) = H^k(\R^N)$ (all derivatives up to order $k$ are in $L^2$), $k>N/2$. Then $M$ belongs to the Wiener algebra $W(\R^N)$, and  $\|M\|_W \le C \|M\|_{W^{k,2}}$, where $C=C(N)$. 
\end{lem} 
\begin{proof}
Let  $\rho$ be the inverse Fourier transform of $M$. 

Condition $M \in W^{k, 2} (\R^N) $ means that $ (1+ |x|^k) \rho \in L^2 (\R^N)$. Then by the Cauchy--Schwarz inequality 
\begin{align*}
\int_{\R^N} |\rho(x)| dx  & \le \left( \int_{\R^N} (1+ |x|^k)^{-2}dx \right)^{1/2} \left( \int_{\R^N} (1+ |x|^k)^{2}|\rho(x)|^2 dx \right)^{1/2} 
\\
&  
\le C(N) \|M\|_{W^{k,2}}. 
\end{align*}
So $M$ belongs to the Wiener algebra. 
\end{proof}

Note that the sufficient condition $M\in W^{k,2}$ is far from necessary: While the Wiener algebra is scale invariant, i.e.~the scaling operator $S_a$, $S_a f(x) = f(ax)$ is an isometry in $W$, one can easily see that the operator norm $\|S_a\|_{W^{k,2}\to W^{k,2} } \to \infty$ as $a\to 0$. 

The next lemma, an analog of Lemma \ref{l:wiener-schur-1}, gives a sufficient condition for Schur multipliers that are not translation invariant (i.e.~not of the form $M(t-s)$). 

\begin{lem}
\label{l:wiener-schur-2}
Let $M\in W(\R^N\times \R^N) = W(\R^{2N})$. Then $M$ is a Schur multiplier on the set of restrictedly bounded singular kernels in $\R^N$ with Schur norm at most $\|M\|_{W(\R^{2N})}$. 
\end{lem}

\begin{proof}
Let $p$, $\mu$, $\nu$ be fixed, and let $K$ be an $L^p$ restricted singular kernel on $R^N$. Since for $a\in \R$ the multiplication by $e^{ia x}$ is an isometry in $L^p(\mu)$ (and in $L^{p'}(\nu)$), the kernel $e^{-i a \cdot s} K(s, t) e^{-i b \cdot t}$ has the same $L^p$ restricted bound as the kernel $K$. 

Because $\rho \in L^1(\R^N\times \R^N)$, one can immediately see that the restricted norm of the ``averaged'' kernel $\wt K$, 
\[
\wt K(s, t):= \int_{\R^N\times \R^N} K(s, t) \rho(a, b) e^{-ia\cdot s} e^{-i b\cdot t} da db = K(s, t) \widehat \rho (s, t)
\]
is at most $\|\widehat \rho \|_{W} \|K\|\ti{restr}$, where $\|K\|\ti{restr}$ is the ($L^p$, $\mu$, $\nu$) restricted norm of $K$. 
\end{proof}

\begin{rem}
In the above Lemma \ref{l:wiener-schur-2} one can replace class $W=W(\R^N\times \R^N) $ by the class $\widehat\cM$ of the Fourier transforms of the measures (charges) of bounded variation in $\R^N\times \R^N$. 

This point of view unifies the situations described in Lemmas \ref{l:wiener-schur-1} and \ref{l:wiener-schur-2}. Namely, if one considers on the $N$-dimensional subspace  $D=\{ (-x, x): x\in \R^N\}$ the measure (not necessarily non-negative)  $\sigma$, $d\sigma= \rho(x)dx $ in the parametrization $x\mapsto (-x,x)$, $x\in\R^N$, then the Fourier transform $\widehat \sigma$ of $\sigma$ (treated as a measure on the whole $\R^N\times \R^N$) is exactly $\widehat\rho(t-s)$. 
\end{rem}

\begin{rem}
One can use Lemma \ref{l:sob-wiener} for a sufficient condition for $M\in W(\R^N\times \R^N)$; in this case it is sufficient that $M\in W^{k,2}(\R^{2N})$, $k>N$. 
\end{rem}

\begin{rem*}
As we already mentioned, the Wiener algebra is scale invariant, i.e.~the functions $M$ and $M_\e$, $M_\e(x):= M(x/\e)$ have the same norm in the Wiener algebra (a well known fact and an easy exercise).  The same, of course, holds for the space $\widehat \cM$. 

Therefore, if $M\in \widehat \cM$, and $M_\e(x) :=M(x/\e)$, then all $M_\e$ are Schur multipliers with the uniform estimate on the Schur norm. 

This fact will be exploited a lot in this paper.
\end{rem*}

Let us now state a simple corollary. 

\begin{cor}
\label{cor:Schur-mult}
Let $m$ be in the Sobolev space $H^k(\R^N)$, $k>N/2$. Then the kernels $K_\e$, $K_\e(s, t) := m((t-s)/\e)$ are Schur multipliers with uniformly bounded Schur norms.
\end{cor}

\subsection{Smooth mollifying multipliers}
\label{s:smm}

We can summarize the above discussion in the following
 proposition. 

\begin{prop}
\label{p:smm}
Let $m$ be a function on $\R^N$ such that $m\equiv 0$ in a neighborhood of $0$, and $1-m\in H^k(\R^N) =W^{k,2}(\R^N)$, $k>N/2$. Then the functions $M_\e$, 
\begin{align*}
M_\e(s, t) := m((t-s)/\e)
\end{align*}
is the family of smooth regularized multipliers, described above in Section \ref{ss:main-results}, meaning that

\begin{enumerate}
\item $M_\e (s, t)\to 1$ as $\e\to 0$ uniformly on all sets $\{s, t \in\R^N: |s-t| >a\}$, $a>0$; 
  \item For any singular kernel $K$ (with respect to Radon measures $\mu$ and $\nu$) the regularized kernels $K_\e = K M_\e$ are locally in $L^2(\mu\times\nu)$, so the corresponding operators are well defined for bounded compactly supported functions;
    \item If the kernel is restrictedly bounded in $L^p$ (i.e.~if the estimate \eqref{restr-bound} holds for all bounded $f$, $g$ with separated compact supports), and the measures $\mu$ and $\nu$ do not have common atoms, then the regularized integral operators $T_\e$ with kernels $K_\e$ are uniformly (in $\e$) bounded as operators $L^p(\mu) \to L^p(\nu)$. 
\end{enumerate}
\end{prop}

\section{Restricted boundedness implies boundedness} 
 
In this section we are going to show that for kernels which are locally $L^2(\mu\times\nu)$ (and also for non-negative kernels) the restricted $L^p$ boundedness with restricted norm $C$ implies boundedness with the norm at most $2C$. 

The main application of this result is as follows. Suppose $K$ is an $L^p$ restrictedly bounded singular kernel. Multiplying $K$ by a smooth mollifying multiplier, described above in Section \ref{sec2}, see for example Proposition \ref{p:smm}, we will get a family of uniformly (in $\e$) restrictedly bounded kernels $K_\e:= M_\e K$. The kernels $K_\e$ are locally $L^2(\mu\times\nu)$, so by the main result of this section the corresponding operators $T_\e$ are uniformly bounded operators from $L^p(\mu)$ to $ L^p(\nu)$. 

Let $T$ be a limit point of $T_\e$, $\e\to 0$ in the weak operator topology. 
Statement \cond1 of Proposition \ref{p:smm} will imply that for $f$ and $g$ with separated compact supports 
\begin{align*}
\La Tf, g\Ra_\nu = \int K(s, t) f(t) g(s) \, d\mu(t) d\nu(s) , 
\end{align*}
so the limit operator $T$ can be indeed considered as a singular integral operator with kernel $K$.

The statement about non-negative kernels will be used to show that under some additional assumptions about kernel $K$ (which do not involve non-negativity),  the restricted boundedness of $K$ implies that the measures $\mu$ and $\nu$ satisfy two weight Muckenhoupt condition, see Theorem \ref{t:Muckehnoupt} below. 

\subsection{Separated partitions of unity}

\begin{lem}
\label{l:split1}
Let $\sigma$ be a Radon measure without atoms in $\R^N$. There exist Borel sets $E^1_n$, $E^2_n$, $n\in\N$ such that  
\begin{enumerate}
    \item For all $n\in \N$ the sets $E_n^1$ and $E_n^2$ are separated, i.e.~$\dist(E_n^1, E_n^2) > 0$. 
    
    \item The operators $P_n^k$, $P_n^k f := \1_{E^k_n} f$, $k=1, 2$ converge to $\frac12 I$ in weak operator topology in $L^2(\sigma)$. 
    
    \item For any $p\in [1, \infty)$ and for $k=1, 2$
    \[
    \lim_{n\to\infty} \|\1_{E_n^k} f \|_{L^p(\sigma)} = 2^{-1/p} \| f \|_{L^p(\sigma)}, \qquad \forall f\in L^p(\sigma). 
    \]
\end{enumerate}
\end{lem}

\begin{defn*}
The standard grid $G$ of size $r$ in $\R^N$ is the collection of cubes $r \bj + [0, r)^N$, $\bj\in \Z^N$. A grid of size $r$ is a translation of the standard grid by some $a\in\R^N$.

The boundary $\partial G$ of a grid $G$ is the union of all boundaries $\partial Q$, $Q\in G$. 
\end{defn*}

For a cube $R= x + [0,r)^N$, $x\in \R^N$, and $\tau \in\R$ let $\tau R$ denote its dilation, $\tau R := x+ [0, \tau r)^N$. Note, that the cube is dilated with respect to its corner, not its center. 

\begin{lem}
\label{l:s-count}
Let $\sigma$ be a Radon measure in $\R^N$. For a cube $R= x + [0,r)^N$  let $\tau R$ be its dilation as described above. Then 
\[
\lim_{\tau\to 1} \sigma(\tau R) = \sigma (R). 
\]
\end{lem}

\begin{proof}
The statement of the lemma is a trivial corollary of countable additivity.  Recall that for finite measures countable additivity is equivalent to the fact that $\sigma(\cup_{k\ge 0} R_k) = \lim_{k\to \infty}\sigma (R_k)$ for any increasing sequence of measurable sets $R_k$. Since the family $\tau R$, $\tau>0$, is an increasing (with respect to $\tau$) family of cubes, and since $\cup_{\tau\in(0,1)} \tau R = R$, we get that
\[
\lim_{\tau\to 1-} \sigma(\tau R) = \sigma (R).
\]

Since in what follows we only need this identity, we leave the rest of the lemma (i.e.~the case of $\lim_{\tau\to 1+}$) as an easy exercise to the reader, who just needs to recall the restatement of countable additivity in terms of decreasing sequences of sets. 
\end{proof}

\begin{proof}[Proof of Lemma \ref{l:split1}]
Since for any sets $E^k_n$ the operators $P_n^k$, $P_n^k f := \1_{E^k_n} f$ are contractions in all $L^p(\sigma)$, to prove \cond2 it is sufficient to show that 
\[
\lim_{n\to\infty} \langle P^k_n f, g \rangle = \frac12 \langle f, g \rangle 
\]
for $f$ and $g$ in some dense in $L^p(\sigma)$ set. In particular, it is sufficient to prove this identity for $f$ and $g$ being the finite sums $\sum_{j}\alpha_j \1_{Q_j}$, where $Q_j$ are some (standard) dyadic cubes. 

Because of the continuity of the norm, it is also sufficient to check condition \cond3 on a dense set,  for example again on the finite sums  $\sum_{j}\alpha_j \1_{Q_j}$. 

So, it is sufficient to show that for any standard dyadic cube $Q$ 
\begin{equation}
\label{half1}
\lim_{n\to\infty} \sigma( E_n^k \cap Q ) = \frac12 \sigma(Q), \qquad k=1, 2. 
\end{equation}

To prove \eqref{half1} we will construct the sets $E_n^k$ in such a way that 
\begin{equation}
\label{half2}
\left| \sigma(E_n^k \cap Q) -  \sigma(Q)/2 \right| < 2^{-n} \sigma(Q) , \qquad k=1, 2
\end{equation}
for all standard dyadic cubes of size $2^{-n}$ inside the large cube $Q^n:=[-2^n , 2^n)^N$. Trivially, the same estimate will hold for all larger dyadic cubes inside $Q^n$, which trivially implies \eqref{half1}. 

Let us construct the sets $E_n^k$. For each $n$ we will first construct the sets $\wt E_n^k$, such that the sets $\wt E^1_n$ and $\wt E_n^2$ are disjoint (but not necessarily separated), and \eqref{half2} is satisfied. The sets $\wt E_n^k$ will be constructed as unions of the (small) standard dyadic cubes, and by shrinking each cube a little, we will get separated sets $E_n^1$ and $E_n^2$. 

Let $\alpha =\min_Q \sigma (Q)$, where the minimum is taken over all all standard dyadic cubes $Q\subset Q^n$ of size $2^{-n}$ for which $\sigma(Q)\ne 0$. 
Pick $\delta_0>0$ such that 
\[
\sigma(R) < 2^{-n}\alpha
\]
for all cubes $R\subset Q^n$ such that $\ell(R)<\delta_0$ (recall, that $\ell(R)$ is the \emph{size}, i.e.~the sidelength of the cube $R$). Note, that clearly $\delta_0<2^{-n}$. Pick some $\delta<\delta_0$ of form $\delta = 2^{-m}$, $m\in \Z$. 

Let us split the cube $Q^n$ into the standard dyadic cubes of size $\delta$, and construct the sets $\wt E^1_n$ and $\wt E^2_n$ as the finite unions of such cubes. Namely, for each dyadic cube $Q\subset Q^n$, $\ell(Q) = 2^{-n}$ we distribute the dyadic cubes $R\subset Q$, $\ell(R) =\delta$ between the sets $\wt E^1_n$ and $\wt E^2_n$ as follows: 

We assign the first such cube $R$ to be in $\wt E^1_n$, the second one to be in the set $\wt E^2_n$, and then on the each next step we will add a cube to the set of smaller measure $\sigma$ (in the case when both sets have the same measure, we can add the next cube to either of the sets, say to $\wt E^1_n$ for definiteness). We stop when all such cubes $R$ are exhausted, and then repeat the procedure for the other cubes $Q$. 

Since by the choice of $\delta$  for each dyadic cube $Q\subset Q^n$, $\ell(Q) = 2^{-n}$ we have
\[
\sigma(R) <2^{-n} \alpha \le 2^{-n} \sigma(Q),
\]
and since on each step we add such cube $R$ to a set of smaller (or equal) measure, we can conclude that 
\[
\left| \sigma(\wt E_n^1 \cap Q) -  \sigma(\wt E_n^2 \cap Q) \right| < 2^{-n} \sigma(Q) , 
\]
which in turn implies (because $\sigma(\wt E^1_n\cap Q) + \sigma(\wt E^2_n\cap Q) =\sigma(Q)$) that for $k=1, 2$
\[
\left| \sigma(\wt E_n^k \cap Q) -  \sigma(Q)/2 \right| < 2^{-n-1} \sigma(Q).
\]

We then obtain the sets $E^1_n$ and $E^2_n$ by replacing each dyadic cube $R$, $\ell(R)=\delta$ in the sets  $\wt E^1_n$ and $\wt E^2_n$ by the cube $\tau R$, where $\tau \in (0,1)$ is sufficiently close to $1$. Clearly, for any $\tau\in (0,1)$ the sets $E^1_n$ and $E^2_n$ are separated. Moreover, Lemma \ref{l:s-count} assures that by picking $\tau$ sufficiently close to $1$ we can make the differences $\sigma(\wt E^k_n) - \sigma( E^k_n)$  as small as we want, so if $\tau$ is sufficiently close to $1$ then the estimate \eqref{half2} holds.  
\end{proof}

Let us now consider the general case. For a measure $\mu$ let $\mu\ti c$ and $\mu\ti a$ be the continuous and purely atomic parts of $\mu$ respectively, $\mu= \mu\ti c + \mu\ti a$. For a $\mu$-measurable function $f$ consider the decomposition $f=f_{\mu\ti c} + f_{\mu\ti a}$, where $f_{\mu\ti c}$ and $f_{\mu\ti a}$ are the projections of $f$ onto continuous and atomic parts of $\mu$ respectively, 
\[
f_{\mu\ti a}(x) = \left\{
\begin{array}{ll}
f(x), \qquad & \mu(\{x\}) >0; \\
0, & \mu(\{x\}) =0 .
\end{array}
\right.
\]

\begin{cor}
\label{c:split2}
Let $\mu$ and $\nu$ be Radon measures in $\R^N$ without common atoms. There exist Borel sets $E^1_n$, $E^2_n$, $n\in\N$ such that  
\begin{enumerate}
    \item For all $n\in \N$ the sets $E_n^1$ and $E_n^2$ are separated, i.e.~$\dist(E_n^1, E_n^2) > 0$. 
    
    \item The operators $P_n^1$ and $P_n^2$, $P_n^1 f := \1_{E^1_n} \left(f_{\mu\ti c} + \frac12 f_{\mu\ti a}\right)$, $P_n^2 g := \1_{E^2_n} \left(g_{\nu\ti c} + \frac12 g_{\nu\ti a}\right)$ converge to $\frac12 I$ in weak operator topology of $L^2(\mu)$ and $L^2(\nu)$ respectively.  
    
    \item For any $p\in [1, \infty)$ and for any $f\in  L^p(\mu)$, $g\in L^{p}(\nu)$
    \[
    \lim_{n\to\infty} \|\1_{E_n^1} f \|_{L^p(\mu)} \le 2^{-1/p} \| f \|_{L^p(\mu)},\quad 
    \lim_{n\to\infty} \|\1_{E_n^2} g \|_{L^p(\nu)} \le 2^{-1/p} \| g \|_{L^p(\nu)}. \qquad 
    \]
\end{enumerate}
\end{cor}
\begin{proof}
Take $\sigma := \mu\ti c + \nu\ti c$, where $\mu\ti c $ and $\nu\ti c$ are continuous parts of the measures $\mu$ and $\nu$ respectively. 
Then 
\begin{equation}
d \mu\ti c = w d\sigma, \qquad d\nu\ci c = v d\sigma, 
\end{equation}
where $w, v\ge 0$ are some weights, $w+v\equiv 1$. 

If the measures $\mu$ and $\nu$ do not have atoms, then the  sets $E_n^1$ and $E_n^2$ from the above Lemma \ref{l:split1} are exactly the sets we need. Indeed, for bounded functions $f$ and $g$ statements \cond2 and \cond3 of the corollary (weak convergence and convergence of norms in $L^p(\mu)$ and $L^p(\nu)$) follow from the corresponding statements of Lemma  \ref{l:split1} (convergence in $L^p(\sigma)$). Since the operators $P^k_n$ are contractions for any choice of the sets $E_n^k$, one can use $\e/3$-theorem to extend the statements \cond2 and \cond3 from  a dense sets of bounded functions to all $L^p(\mu)$ ($L^p(\nu)$). 

For the general case, 
let 
\[
\mu\ti{a} = \sum_{n} \alpha_n \delta_{x_n}, \qquad \nu\ti{a} = \sum_{n} \beta_n \delta_{y_n}
\]
be the purely atomic parts of the measures $\mu$ and $\nu$ respectively. Without loss of generality, assume that the sequences $\{\alpha_n\}$ and $\{\beta_n\}$ are non-increasing. Note that since the measures $\mu$ and $\nu$ do not have a common atom, $x_n\ne y_k$ for all $n$ and $k$.

Let $\wt E_n^1$ and $\wt E_n^2$ denote  the sets obtained in Lemma \ref{l:split1} for $\sigma := \mu\ti{c} +\nu\ti c$ (we use the notation $\wt E^k_n$ instead of $E^k_n$ because we want to use the notation $E^k_n$ for the final ``output''). We can always assume without loss of generality that $x_j, y_j\notin \wt E^k_n$ (for all $j$, $n$ and $k$). Define
\[
E^1_n := \Bigl(\wt E^1_n \cup \bigcup_{j=1}^n x_j\Bigr) \setminus \bigcup_{j=1}^\infty B(y_j, r^{(n)}_j), \qquad 
E^2_n := \Bigl(\wt E^2_n  \cup \bigcup_{j=1}^n y_j \Bigr) \setminus \bigcup_{j=1}^\infty B(x_j, r^{(n)}_j ), 
\]
where for each $n$ the radii $r^{(n)}_j$ are picked in such a way that 
\[
\sum_{j=1}^\infty \left[ \sigma (B(x_j, r^{(n)}_j)) + \sigma (B(y_j, r^{(n)}_j)) \right] < 2^{-n} 
\]
and such that $ \bigcup_{j=1}^n x_j \cap \bigcup_{j=1}^\infty B(y_j, r^{(n)}_j)=\bigcup_{j=1}^n y_j  \cap \bigcup_{j=1}^\infty B(x_j, r^{(n)}_j ) =\varnothing$.

Let us show that $E_n^k$ are the desired sets. Let us decompose $f\in L^p(\mu)$ as $f= f\ti c + f\ti a$, where $f\ti a : = f \1\ci{\cup_{j=1}^\infty \{ x_j\}} $ is the purely atomic part of $f$. Clearly  
\[
\lim_{n\to \infty} \| f\ti a - \1\ci{E^1_n} f\ti a \|\ci{L^p(\mu)} = 0 ,  
\]
so the statements \cond2 and \cond3 of the corollary hold for purely atomic functions $f\ti a$. Note, that unlike Lemma \ref{l:split1} we have the inequality in statement \cond3 here, because
\[
\lim_{n\to \infty}  \| 2^{-1}\1\ci{E^1_n} f\ti a \|\ci{L^p(\mu)} = 2^{-1} \|f\ti a\|\ci{L^p(\mu)} \le 2^{-1/p} \|f\ti a\|\ci{L^p(\mu)} ,  
\]

We can also estimate 
\begin{equation}
\label{appr1}
\| f\ti c \1\ci{ E^1_n } - f\ti c \1\ci{\wt E^1_n } \|\ci{L^p(\mu_{\text c})}^p
\le \int\limits_{\bigcup_{j=1}^\infty B(y_j, r^{(n)}_j)} |f\ti c|^p d\mu\ti c \to 0 \qquad \text{as } n\to \infty . 
\end{equation}
because 
\[
\mu\ti c\Bigl( \bigcup_{j=1}^\infty B(y_j, r^{(n)}_j ) \Bigr) \le \sigma \Bigl(\bigcup_{j=1}^\infty B(y_j, r^{(n)}_j) \Bigr) \le 2^{-n}. 
\]
Therefore, since   the statements \cond2 and \cond3 of the corollary hold for  the sets $\wt E^1_n$ and the measure $\mu\ti c$ (because, as we discussed above, the sets from Lemma \ref{l:split1} work for measures without atoms), equality \eqref{appr1} implies that these statements hold for the sets $E^1_n$ and the measure $\mu\ti c$ as well. 

So, we can say that  the statements \cond2 and \cond3 of the corollary hold for $f\ti a$ and $f\ti c$ (and the measure $\mu$), and therefore they are true for $f$. 

The statements for the measure $\nu$ can be proved absolutely the same way. 
\end{proof}

\subsection{Uniform boundedness}

\begin{theo}
\label{t:restr-bd1}
Let $\mu$ and $\nu$ be Radon measures in $\R^N$ without common atoms. Assume that a kernel $K\in L^2\ti{loc} (\mu\times \nu)$ is $L^p$ resrictedly bounded, with the restricted norm $C$. Then the integral operator with $T$ kernel $K$ is a bounded operator $L^p(\mu)\to L^p(\nu)$ with the norm at most $2C$.  
\end{theo}

\begin{proof}
Let $f$, $g$ be (Borel measurable) functions, supported on a cube $Q$.   Let us restrict everything to the cube $Q$. Then the integral operator $T$ with kernel $K$ is a Hilbert--Scmidt (and therefore compact) operator, $T: L^2(Q, \mu)\to L^2(Q, \nu)$. 

Let $P^k_n$ be the projections from Corollary \ref{c:split2}. Then by statement \cond2 of the corollary (weak convergence) and because $T$ is compact, 
\[
\lim_{n\to\infty} \La T P^1_n f , P^2_n g \Ra = \frac14 \La T f, g\Ra. 
\]
On the other hand, 
by restricted $L^p$ boundedness
\begin{equation}
\label{r-bound3}
\left| \La  T P^1_n f , P^2_n g \Ra \right| \le C \|  P^1_n f \|_{L^{p}(\mu)} \|  P^2_n g \|_{L^{p'}(\nu)} 
\end{equation}
and by statement \cond3 of Corollary \ref{c:split2} 
\[
\lim_{n\to\infty}
\|  P^1_n f \|_{L^{p}(\mu)} \|  P^2_n g \|_{L^{p'}(\nu)} = 2^{-1/p} \|   f \|_{L^{p}(\mu)}   
2^{-1/p'}\|   g \|_{L^{p'}(\nu)}. 
\]
So, taking limit in both sides of \eqref{r-bound3} we get that 
\[
 \frac14 |\La T f, g\Ra | \le \frac12  \|   f \|_{L^{p}(\mu)} \|   g \|_{L^{p'}(\nu)}.
\]
which is exactly the desired estimate. Since it holds on a dense set of bounded compactly supported functions, we are done. 
\end{proof}

The next result is an easy corollary of the previous theorem. 
\begin{theo}
\label{t:restr-bd2}
Let $\mu$ and $\nu$ be Radon measures in $\R^N$ without common atoms. Assume that a kernel $K\ge 0$ is $L^p$ resrictedly bounded, with the restricted norm $C$. Then the integral operator with $T$ kernel $K$ is a bounded operator $L^p(\mu)\to L^p(\nu)$ with the norm at most $2C$.  
\end{theo}
\begin{proof}
First, let us notice, that the integral operator with kernel $K\ge0$ is well defined for $f\ge0$, and to compute its norm we only need to test it on $f\ge0$. 

Second, the norm of this operator can be computed as the supremum (or limit as $R\to\infty$) of the operators with truncated kernels
\[
K\ci R(s, t) = \min\{K(s, t), R\}, \qquad R>0.
\]
If $C$ is the restricted norm of $K$, then $C$ is also a restricted bound for all of $K\ci R$. But kernels $K\ci R$ are bounded, so by Theorem \ref{t:restr-bd1} the corresponding integral operators are bounded with the norm at most $2C$. Taking limit as $R\to \infty$ we get the conclusion of the theorem. 
\end{proof}

\subsection{How to treat common atoms}
If the measures $\mu$ and $\nu$ do have common atoms, then the above Theorem \ref{t:restr-bd1} cannot be applied directly. However, using this theorem, it is quite easy to define the boundedness of the singular integral operator in this case. 

Namely, consider the decompositions 
\[
\mu =\wt\mu + \mu_0, \qquad \nu= \wt\nu + \nu_0, 
\]
where $\mu_0$ and $\nu_0$ are the parts of $\mu$ and $\nu$ supported on their common atoms. Then the measures $\mu$ and $\wt \nu$ do not have common atoms; the same is true for $\wt\mu$ and $\mu$. Therefore, to check the $L^p$ boundedness of a singular integral operator with kernel $K$ as an operator $L^p(\mu) \to L^p(\wt\nu)$ or $L^p(\wt\mu)\to L^p(\nu)$. But since the measures do not have common atoms, these operators can be checked using Theorem \ref{t:restr-bd1}. 

So, to check the boundedness of the whole operator, it remains to check the block acting $L^p(\mu_0)\to L^p(\nu_0)$. But the bilinear form of this block is well defined for functions supported at finitely many points (note, that the kernel $K(x, x)$ has to be defined at common atoms of the measures $\mu$ and $\nu$), so there is no problem defining this block.

\section{Uniform boundedness of truncations}

In this section we will show that under some additional assumptions, which are satisfied for the classical operators like generalized  Riesz Transforms (treated as a vector-valued transformation), or Ahlfors--Beurling operator, the restricted $L^p$ boundedness implies the uniform boundedness of the truncated operators $T_\e$, 
\begin{equation}
\label{trunc}
T_\e f (s) = \int_{|s-t|<\e} K(s, t) f(t) d\mu(t). 
\end{equation}

We will need the following definition. 

\begin{defn}
\label{d:sect}
Let $\kappa\in(0,1]$. A function $F$ with values in $\R^d$ is called $\kappa$-sectorial if there exists $x_0\in\R^d$, $|x_0|=1$, such that $\La F(s), x_0\Ra \ge\kappa |F(s)|$ for all $s$ in the domain of $F$. 
\end{defn}

\begin{prop}
\label{p:trunc1}
Let a singular kernel $K$ with values in $\R^k$ be restrictedly bounded in $L^p$ (we assume the measures $\mu$ and $\nu$ in $\R^N$ without common atoms to be fixed). 

Suppose there exists $\delta>0$, $\kappa>0$ and a family of (matrix-valued) Schur multipliers $M_\e(s, t)$, $\e\in(0, \infty)$ with uniformly bounded Schur norm, such that 
\begin{enumerate}
\item for each $\e>0$ the function $M_\e K$ is $\kappa$-sectorial on the set $\{s, t\in \R^N: (1-\delta)\e <|s-t|< \e\}$; 
\item  $| M_\e K| \ge |K|$ on the set   $\{s, t\in \R^N: (1-\delta)\e <|s-t|< \e\}$.  
\end{enumerate}

Then the truncated operators $T_\e$ defined by \eqref{trunc} are uniformly (in $\e$) bounded operators $L^p(\mu)\to L^p(\nu)$. 
\end{prop}

\begin{proof}
Take a function $m\in C^\infty(\R)$ such that $m(x)=1$ for $x\ge 1$ and $m(x) = 0$ for $x\le 1-\delta$. 
Then  the function $\wt m$, $\wt m(s) = m(|s|)$, $s\in \R^N$ satisfies $1-\wt m\in C^\infty_0(\R^N)$. 
As it was discussed in section \ref{sec2} above, this implies that the functions $m(|s-t|/\e) =\wt m((s-t)/\e)$ are Schur multipliers with uniform (in $r$) estimate on the Schur norm. 

Therefore, the smoothly regularized kernels $m(|s-t|/\e) K(s, t)$ give a uniformly bounded family of operators. The difference between the kernel of the truncated operator $T_\e$ and the kernel $m(|s-t|/\e) K(s, t)$ is given by $\psi(|s-t|/\e) K(s, t)$, where $\psi (x) = m(x)-\1_{[1, \infty)}(x)$. Thus, to prove the uniform boundedness of the truncated operators, it is sufficient to prove the uniform boundedness of the operators with kernels $\psi(|s-t|/\e) K(s, t)$. 

We now use the trivial observation, that if $T_1$, $T_2$ are integral operators between $L^p$ spaces (defined initially on dense sets) with kernels $K_1$ and $K_2$ respectively, and if 
$|K_1|\le K_2$, then the boundedness of $T_2$ implies the boundedness of $T_1$ and the estimate $\|T_1\|\le \|T_2\|$.
Therefore, since 
\[
|\psi(|s-t|/\e) K(s, t) | \le \chi(|s-t|/\e) |K(s, t)|
\]
where $\chi := \1_{[1-\delta, 1]}$, to prove the proposition it is sufficient to show that the operators with kernels $\chi(|s-t|/\e) |K(s, t)|$ are uniformly (in $\e$) bounded. 

Let $M_\e$ be the Schur multiplier from the assumption of the proposition, and let $x_0\in\R^d$ be such that 
\begin{equation}
\label{sect1}
\La M_\e(s, t) K(s, t) , x_0\Ra \ge \kappa | K(s, t) | \qquad \forall s, t\in\R^N: \ (1-\delta)\e <|s-t|<\e. 
\end{equation}

The operators with (vector-valued) kernel $M_\e(s, t) K(s, t)$ are uniformly (in $\e$) bounded, and therefore so are the operators with the scalar-valued kernels $\La M_\e(s, t) K(s, t) , x_0\Ra $. 

The estimate \eqref{sect1} implies that 
\[
\chi(|s-t|/\e) |K(s, t)|\le\kappa^{-1} \La M_\e(s, t) K(s, t) , x_0\Ra, 
\]
and thus the operators with kernels $\chi(|s-t|/\e) |K(s, t)|$ are uniformly bounded. 
\end{proof}

\subsection{Some examples}

\begin{ex}
Consider a convolution kernel $K(s, t) =K_1(t-s)$,  
\begin{equation}
\label{eq:c-kern1}
K_1 (x) =A(|x|) B(x/|x|), 
\end{equation}
where $A(r)\ge 0$ for all $r>0$ and $B$ is a function (with values in some $\R^m$) in the Sobolev space $H^k$, $k>N/2$  on the unit sphere $\s_{N-1}$ in $\R^N$. 

If $B(s)\ne 0$ for all $s\in \s_{N-1}$, then the kernel $K$ satisfies the assumption of Proposition \ref{p:trunc1}. 

Indeed, let 
\[
C=\max_{s\in \s_{N-1}} |B(s)|^{-1}, 
\]
and let $\f\in C^\infty_0(\R)$ be such that $\f\equiv 1$ on $[0.9, 1]$, and $\f(x) \equiv 0$ for $x\notin (0.8, 1.1)$. 
Then the function $m$, $m(s) = C \f(s) B^T(s/|s|)$, where $B^T$  stands for the transpose of $B$, is clearly in $H^k(\R^N)$. 

Therefore, by Corollary \ref{cor:Schur-mult}
the functions  $M_r(s, t)$ defined by
\[
M_r (s, t) :=  m((t-s)/r) = C\f(|s-t|/r) B^T((t-s)/|t-s|)
\]
are Schur multipliers  with uniformly bounded Schur norms.  It is trivial to see that the assumptions of \cond1 and \cond2 of Proposition \ref{p:trunc1} are also satisfied (with $\kappa=1$ and $\delta =0.1$).

Thus, for any such kernel, the restricted $L^p$ boundedness implies the uniform boundedness of the truncated operators defined by \eqref{trunc}.

\end{ex}

Examples of such kernels include the kernel of the (vector-valued) generalized Riesz Transform in $\R^N$ ($K_1(s) = s/ |s|^{\alpha+1}$, $s\in\R^N$, $\alpha >0$), or the Cauchy  ($K_1(z) = 1/z= \overline z /|z|^2$) and Ahlfors--Beurling  ($K_1(z) =1/z^2 = \overline z^2/|z|^4$) transforms in the complex plane.

Note that the classical Riesz transform is a particular case ($\alpha = N$) of the generalized one. 

Another example is given by the Beurling--Ahlfors transform $\cS$ on forms, defined, for example in \cite{Iwaniec-Martin_RieszT_1996}. The fact that the kernel of $\cS$ admits the representation \eqref{eq:c-kern1} can be easily seen from formula (112) on p.~53 of \cite{Iwaniec-Martin_RieszT_1996}.

\begin{rem}
Note that Proposition \ref{p:trunc1} cannot be applied to individual (coordinate) generalized Riesz transforms ($K_1(s) = s_k /|s|^{\alpha +1}$, $s=(s_1, s_2, \ldots, s_N)^T$). While the restricted boundedness of such individual  kernel implies the uniform boundedness of its smooth regularizations, we do not know whether it implies the uniform boundedness of the truncations \eqref{eq:c-kern1}. 

We suspect that the answer here is negative.
\end{rem}

\section{Two weight Muckenhoupt condition}

\begin{theo}
\label{t:Muckehnoupt}
Let $K$ be a (vector-valued) convolution kernel in $\R^N$, $K(s, t)= K_1(t-s)$, 
\begin{align}
\label{eq:k-conv2}
    K_1(s) = A(|x|) B(x)
\end{align}
where $B\in H^k(\R^N)=W^{k,2}(\R^N)$, $k>N/2$  is a homogeneous of order $d>0$ function, 
\[
B(cx) = c^d B(x) \qquad \forall x\in\R^N\ \forall c\in \R_+, 
\]
bounded away from $0$ on the unit sphere, 
and $A$ is a function on $\R_+$ such that for some $\alpha\ge d$
\begin{align}
\label{eq:x^{-d-alpha}}
A(x)\ge x^{-d-\alpha}\qquad \forall x\in\R_+. 
\end{align}

Let $\mu$ and $\nu$ be Radon measures without common atoms in $\R^N$, and let the kernel $K$  be  $L^p$ restrictedly bounded (with respect to the measures $\mu$ and $\nu$).

 Then the measures $\mu$ and $\nu$ satisfy the following generalized two-weight Muckenhoupt  $A_p^\alpha$ condition of order $\alpha$;
\begin{equation}
\label{A_p^alpha}
\tag{$A_p^\alpha$}
\sup_{B}\, (\diam B)^{-\alpha} \mu(B)^{1/p'} \nu(B)^{1/p} <\infty;  
\end{equation}
here the supremum is taken over all balls in $\R^N$.
\end{theo}

Examples of kernels satisfying assumption of this proposition are generalized vector-valued Riesz transform ($d=1$), Cauchy transform in $\C$ ($\alpha=1$, $d=1$), Beurling--Ahlfors transform in $\C$ ($\alpha=2$, $d=2$).

The \emph{classical} Muckenhoupt condition ($\alpha =N$) is well-known in analysis: for classical (one weight) weighted estimates ($d\mu= w dx$, $d\nu = w^{-1}dx$) it is well-known that the classical Muckenhoupt condition $A_p =(A_p^N)$ ($\alpha =N$) is necessary and sufficient for the $L^p$ boundedness of classical singular integral operators, like Hilbert Transform, or vector Riesz Transform. 

In the case of one measure ($\mu=\nu$) the condition $(A_p^\alpha)$ is independent of $p\in (1, \infty)$ and is equivalent to the growth estimate $\mu(B(x_0, r)) \le Cr^\alpha$ (where $B(x_0, r)$ stands for the open ball of radius $r$ centered at $x_0$) uniformly in $x_0$ and $r$. 

In the one measure case this growth condition for $\alpha=1$ is known to be necessary (but not sufficient) for the boundedness of the Cauchy Transform in $\C$ in $L^2(\mu)$. It was probably well-known to specialists, although we cannot pinpoint the reference, that in the one measure case the growth condition $\mu(B(x_0, r)) \le C r^\alpha$ is necessary for the boundedness of the vector Riesz Transform in $L^p(\mu)$.

This condition also appeared in more general situations as well. For example, it was shown by the second author, see \cite[p.~318]{ProbBook3_v1} that the condition $(A_2^1)$ (in fact, a bit stronger version, where averages over intervals are replaced by Poisson averages) is necessary for the boundedness of the Hilbert Transform in general two weight situation. As easy examples show, the two weight $A_2$ condition is not sufficient for the $L^2$ boundedness of the Hilbert Transform in the general two-weight case. 

It was conjectured for some time that
the stronger ``Poisson'' $A_2$ condition (which is necessary for the two weight estimate of the Hilbert Transform) is also sufficient for the two weight estimate, but this conjecture was later disproved by F.~Nazarov.   

However, the necessity of this condition for general operators as in the above theorem is completely new, and did not appear in the literature. The only exception here is probably  our paper \cite{mypaper}, where it was shown that the condition $(A_2^1)$ is necessary for the $L^2$ boundedness of rather general singular integral operators $T:L^2(\mu)\to L^2(\nu)$ on the real line. 

For example, in the general two weight case, even the necessity of the condition $(A_2^1)$ for the $L^2$ boundedness of the Cauchy Transform was not known (at least it was not presented in the literature). The same can be said for the condition $(A_p^2)$ for the $L^p$ boundedness of the Beurling--Ahlfors Transform in $\C$. The result for the generalized vectors Riesz transforms of order $\alpha$ is also new. 

\begin{proof}[Proof of Theorem \ref{t:Muckehnoupt}]
Define 
\begin{align*}
m(x) := B(x) \f(|x|), \qquad \forall x\in \R^N, 
\end{align*}
where $\f\in C^\infty_0(\R)$ such that $\f\equiv 1$ on $[0,2]$. Clearly $m\in H^k(\R^N)= W^{k,2}(\R^N)$, so by Corollary \ref{cor:Schur-mult} the functions $M_\e$, $M_\e(s, t) := m((t-s)/\e)$ are Schur multipliers with uniformly bounded Schur norms.  

Since
$
B(x) = |x|^d B(x/|x|^d) \ge C |x|^d 
$, 
we conclude that 
\[
|B(x)| \ge C |x|^d
\]
where 
$
C=\inf\{ |B(x)| :x\in \R^N, |x|=1\}>0.
$
Then 
we can estimate
\begin{align*}
B^T(x/\e) B(x) \ge C^2 |x|^{2d} \e^{-d},
\end{align*}
so for $|s-t|\le 2\e$
\begin{align}
\label{eq:K_eps-lower}
K_\e(s, t) := M_\e^T (s, t) K(s, t) \ge C^2 \e^{-d} |t-s|^{d-\alpha} \ge C' \e^{-\alpha};
\end{align}
here in the last inequality we used the fact that $\alpha\ge d$. Note that $K_\e(s, t)\ge 0$ for all $s, t\in \R^N$.

Since $M_\e$ are uniformly bounded Schur multipliers, the operators $T_\e$ with kernel $K_\e$ are uniformly (in $\e$) restrictedly bounded, and so, by
 Theorem \ref{t:restr-bd2} they are uniformly bounded operators $L^p(\mu)\to L^p(\nu)$. 

Let $\cB=\cB(t_0, \e)$ be the open ball of radius $\e$ centered at $t_0$, and let $T_\e$ be the integral operator with kernel $K_\e$. Then estimate \eqref{eq:K_eps-lower} implies that 
\begin{align*}
T_\e \1\ci\cB (s) \ge C' \e^{-\alpha} \mu(\cB)\qquad \forall s\in \cB. 
\end{align*}
Then integrating over $\cB$ we get
\begin{align*}
\|T\1\ci{\cB}\|\ci{L^p(\nu)} \ge C' \e^{-\alpha} \mu(\cB) \nu(\cB)^{1/p}. 
\end{align*}
Since 
\[
\|\1\ci\cB\|\ci{L^p(\mu)} = \mu(\cB)^{1/p}
\]
and the operators $T_\e$ are uniformly bounded (as operators $L^p(\mu)\to L^p(\nu)$), we get the estimate
\begin{align*}
\e^{-\alpha} \mu(\cB) \nu(\cB)^{1/p} \le C \mu(\cB)^{1/p}
\end{align*}
with $C$ independent of $\cB$ and $\e$. But this estimate is equivalent to the conclusion of the proposition
\end{proof}

\section{Concluding remarks}

The main result of this paper  simplifies, even trivializes, the definition of a singular integral operator if only its kernel is given. This paper  does not offer the replacement of the hard analysis technique used to prove the boundedness of singular integral operators; one still has to do hard work of proving the boundedness. However,  it significantly  simplifies the setup. 

For example, in \cite{NTV-2weight-2008} the authors had to spend a lot of time and effort carefully defining their operators. While this was necessary to state the result in all generality, in all interesting situations the operators were abstract singular integral operators, meaning that there was a locally bounded off the diagonal kernel $K(s, t)$ giving the bilinear form for functions with separated compact supports. For example, this approach would work for the so-called dyadic (or Haar) shifts, which recently attracted much attention, see \cite{Lacey-SharpHaarA2-2009, HPTV-A2_2010}. 

But as we had shown in this paper, such operators can be regularized by smooth mollifying multipliers! 
That means that if the operator is restrictedly bounded, then its ``smooth'' regularizations are uniformly bounded, so from the very beginning we can deal with such regularizations. 

It looks a bit ironic, that while the kernels of such dyadic integral operators are very non-smooth, they can be regularized by smooth multipliers. It would be interesting to find regularizations more adapted to the dyadic structure of such operators.

Next, we should mention that since our Schur multipliers are the universal ones, our approach works for operators from $L^p$ to $L^r$, $r\ne p$ as well. We did not want to overload the paper, so we only considered the case $p=r$ in the text. However, the corresponding general statements and their proofs can be   easily obtained from the corresponding parts in the text by obvious modifications. 

We should also mention, that in the classical situation ($d\mu =d\nu = dx$), and even in the ``classical non-homogeneous" (one measure) situation   ($d\mu = d\nu$), it is known that if we have a bounded in some $L^p$ operator $T$ with \cz kernel $K$ (meaning that the bilinear form for functions with separated compact supports is given by $\int K(s, t) f(t) g(s) d\mu(t) d\nu(s)$), then its truncations $T_\e$ are uniformly bounded, and, moreover, the maximal operator $T^\sharp$ is also bounded. 

However, this is a very non-trivial result, requiring rather strong assumptions (\cz kernels, restrictions on the growth of measure, even in the non-homogeneous case). No analogues of our results for more general situations (two measures, no restriction on the growth) were known before: moreover, we suspect that the statement about maximal operator $T^\sharp$ fails in the general situation.


\providecommand{\bysame}{\leavevmode\hbox to3em{\hrulefill}\thinspace}
\providecommand{\MR}{\relax\ifhmode\unskip\space\fi MR }
\providecommand{\MRhref}[2]{%
  \href{http://www.ams.org/mathscinet-getitem?mr=#1}{#2}
}
\providecommand{\href}[2]{#2}

\end{document}